\documentclass[11pt]{amsart}
\usepackage{amssymb,amsmath,amsthm,comment}
\usepackage[margin=1.1in]{geometry}
\usepackage[colorlinks,linkcolor=blue,citecolor=blue,urlcolor=blue]{hyperref}
\usepackage{booktabs}
\usepackage{graphicx}
\usepackage{microtype}
\usepackage{pgfplots}
\pgfplotsset{compat=1.18}

\newcommand{\Q}{\mathbb{Q}}
\newcommand{\Z}{\mathbb{Z}}
\newcommand{\R}{\mathbb{R}}
\newcommand{\C}{\mathbb{C}}
\newcommand{\OO}{\mathcal{O}}
\newcommand{\fa}{\mathfrak{a}}
\newcommand{\fb}{\mathfrak{b}}
\newcommand{\fp}{\mathfrak{p}}
\newcommand{\eps}{\varepsilon}
\newcommand{\Norm}{\mathrm{N}}
\newcommand{\Prob}{\mathbf{P}}
\DeclareMathOperator{\Res}{Res}
\DeclareMathOperator{\ord}{ord}
\DeclareMathOperator{\rank}{rank}

\theoremstyle{plain}
\newtheorem{thm}{Theorem}[section]
\newtheorem{prop}[thm]{Proposition}
\newtheorem{lem}[thm]{Lemma}
\newtheorem{cor}[thm]{Corollary}
\theoremstyle{definition}
\newtheorem{rem}[thm]{Remark}

\title{Determining norm-Euclidean complex cubic fields}

\author[G. Bagger]{Gustav Kj\ae rbye Bagger}\thanks{}
\address{School of Science, The University of New South Wales Canberra, Australia}
\email{g.bagger@unsw.edu.au}

\author[A. Booker]{Andrew R. Booker}\thanks{}
\address{School of Mathematics, University of Bristol, England}
\email{andrew.booker@bristol.ac.uk}

\author[T. Trudgian]{Tim Trudgian}\thanks{TT is partially supported by ARC Grant  DP240100186.}
\address{School of Science, The University of New South Wales Canberra, Australia}
\email{timothy.trudgian@unsw.edu.au}

\date{\today}

\begin{document}
\begin{abstract}
We prove that Cavallar and Lemmermeyer's list of $58$ norm-Euclidean complex
cubic fields is complete.  The proof combines a classical discriminant bound
with exact certificates of non-Euclideanity.  The unit action reduces their
verification to finite lattice enumeration with rigorous interval arithmetic.
An independent covering computation confirms norm-Euclideanity for the
$58$ fields.
\end{abstract}
\maketitle

\section{Introduction}
\noindent
Let $K$ be a number field of degree $n$. Denote by $\OO$ its ring of integers and $\Norm$ the usual norm map. For $\xi\in K$ define
\[
   M(K,\xi):=\inf\{|\Norm(\xi-\gamma)| : \gamma\in\OO\},
   \qquad
   M(K):=\sup_{\xi\in K}M(K,\xi).
\]
We say that $K$ is \emph{norm-Euclidean} if $M(K,\xi)<1$ for every $\xi\in K$.  Note the strict inequality: $M(K,\xi)=1$ for a single $\xi$ already disproves norm-Euclideanity.
We are interested in \emph{complex cubic} fields, i.e.\ $n=3$ with signature $(r_1,r_2)=(1,1)$; equivalently, $d_K<0$. For the real case see \cite{Bagger} and \cite[\S 5.2]{Lemmermeyer}.  Fix the real embedding and one of the two complex embeddings.  For $\xi\in K$, we write
\[
   \xi_{\C}\in\C \quad \text{and} \quad \xi_{\R}\in\R
\]
for its embedded values.  Extending $\xi\mapsto(\xi_{\C},\xi_{\R})$ by $\R$-linearity identifies
$K\otimes_{\Q}\R$ with $\C\times\R$.  We write these coordinates as
$(x+iy,z)$, so that
\begin{equation}\label{eq:norm-embeddings}
   \Norm\xi=|\xi_{\C}|^2\xi_{\R}=(x^2+y^2)z.
\end{equation}
We shall choose a fundamental unit $\eps$ and the complex embedding so that   $\eps_{\R}>1$ and $\Im\eps_{\C}>0.$
It follows that $\Norm\eps=1$, $|\eps_{\C}|=\eps_{\R}^{-1/2}$, and the
regulator is $R_K=\log\eps_{\R}$.

Davenport~\cite{Dav} proved there are only finitely many norm-Euclidean complex cubic fields. Moreover, the discriminant of such a field satisfies
\begin{equation}\label{eq:vdl}
   |d_K|\le170520 .
\end{equation}
This is van der Linden's \cite[Thm.~5.19]{vdL} sharpening of
Cassels'~\cite{Cas} result; see Remark~\ref{rem:bound}.  Cavallar and Lemmermeyer \cite{CL1,CL2} determined
the norm-Euclidean fields with $|d_K|\le 1600$ and conjectured that the resulting list of $58$ discriminants is complete. We resolve this conjecture below.
\begin{thm}\label{thm:main}
A complex cubic field is norm-Euclidean if and only if its discriminant is one of the $58$ values listed by Cavallar and Lemmermeyer \cite{CL2}, the largest absolute value being $876$.
\end{thm}
To prove Theorem~\ref{thm:main}, we construct an exact, independently checkable certificate that $K$ is not norm-Euclidean for every $|d_K|\le170520$ outside the $58$ cases given in~\cite{CL2}. Since a norm-Euclidean ring is Euclidean and hence a principal ideal domain, only fields with class number $h_K=1$ need
be considered; there are $17283$ such fields with $|d_K|\le170520$. This list is provided by the LMFDB \cite{LMFDB}, whose enumeration of the cubic fields of bounded discriminant and of their class groups is complete and unconditional in this range. In fact, already in~\cite{Bel}, Belabas enumerated cubic fields with $|d_K|\leq 10^{11}$.

We briefly outline how the bound~\eqref{eq:vdl} arises, since future researchers may wish to generalize Cassels' and van der Linden's approach to number fields of higher degree.
\begin{rem} \label{rem:bound}
Cassels \cite[Thm.~5]{Cas} proves that any three linear forms $L_1,L_2,L_3$ of
determinant $\Delta$ with $L_1$ real and $L_3=\overline{L_2}$ admit a point
$(t_1^0,t_2^0,t_3^0)$. This is a rational point in the case of conjugate forms of a cubic field. If so, then  $|L_1L_2L_3|\ge|\Delta|/420$ on the whole class of $(t_1^0,t_2^0,t_3^0)$ modulo
$1$.  Taking the $L_i$ to be the embeddings of $t_1\omega_1+t_2\omega_2+t_3\omega_3$
for an integral basis gives $|\Delta|=\sqrt{|d_K|}$ and $L_1L_2L_3=\Norm$, so
$M(K)\ge\sqrt{|d_K|}/420$ and a norm-Euclidean field has $|d_K|<420^2=176400$.

The constant $420$ can be improved.  Both Cassels and van der
Linden \cite[(5.19)]{vdL} arrive at
\[
   |d_K| \le \frac{1024}{3}
   \frac{(k^{1/2}+1)^{2}(k^{3/2}-1)^{6}}{k^{2}(k-2)^{6}}
   \qquad(k>2),
\]
whose right-hand side is minimised at
\[
   k = \bigl(1+\sqrt3\bigr)^{2} = 7.46410\ldots.
\]
Van der Linden computed the minimum numerically.
Its exact value gives
\[
   |d_K| \le 3^8\left(\frac{5+3\sqrt{3}}{2}\right)^2= 85293+\frac{98415\sqrt3}{2} = 170522.89\ldots.
\]
Since $d_K\equiv0,1\pmod4$ and $d_K<0$, the only admissible value in
$[170520,170522.89\ldots]$ is
$170520$ itself.  Hence \eqref{eq:vdl} with equality allowed.
Cavallar and Lemmermeyer \cite{CL2} quote the bound as $|d_K|<170520$, but the distinction matters only in principle; the difference is the single value $d_K=-170520$, and there is
no cubic field of that discriminant.
\end{rem}

Throughout this paper we shall make use of the following criterion.

\begin{prop}\label{prop:crit}
Let $\alpha,\beta\in\OO$, with $\beta\neq0$. Define
$C=\alpha+\beta\OO$ and
$m(C)=\min\{|\Norm\gamma| : \gamma\in C\}$.
Then
\[
   M\!\left(K,\tfrac{\alpha}{\beta}\right)=\frac{m(C)}{|\Norm\beta|}.
\]
In particular, if $m(C)\ge|\Norm\beta|$ then $M(K)\ge 1$ and $K$ is not
norm-Euclidean.
\end{prop}

\begin{proof}
For $\gamma \in \OO$, we have $\Norm(\alpha/\beta-\gamma)=\Norm(\alpha-\beta\gamma)/\Norm(\beta)$ and $\{\alpha-\beta\gamma : \gamma \in \OO\} = C$. Thus
\begin{align*}
M(K,\alpha/\beta) = & \ \inf\{\vert \Norm(\alpha/\beta - \gamma)\vert : \gamma \in \OO\} \\
= & \ \inf\left\{\left\vert \frac{\Norm(\alpha -\beta\gamma)}{\Norm\beta}\right\vert : \gamma \in \OO\right\} \\
= & \ \frac{1}{\vert \Norm\beta\vert}\inf\{\vert \Norm(\gamma)\vert : \gamma \in C\}.
\end{align*}
Since $\{|\Norm\gamma|:\gamma\in C\} \subseteq \Z_{\geq 0}$, the
infimum is attained and equal to  $m(C)$, as required. The last statement follows directly from
$M(K)\ge M(K,\alpha/\beta)\ge 1$.
\end{proof}

The rest of the paper is concerned with finding a \emph{witness pair}
$(\alpha,\beta)$ satisfying the inequality $m(\alpha+\beta\OO)\ge|\Norm\beta|$.
Such a pair proves that $K$ is not norm-Euclidean and gives the quantitative
bound $M(K)\ge m(\alpha+\beta\OO)/|\Norm\beta|$.  Failure to find a pair has
no implication: norm-Euclideanity requires a separate covering argument.  We
carry out such an argument for the $58$ remaining fields in
Section~\ref{sec:converse} using the methods of Cerri and
Lezowski~\cite{Cer,Lez}.

\begin{rem}
\label{rem:nonzero}
The condition $\alpha\notin\beta\OO$ is essential.  If
$\alpha\in\beta\OO$, then $0\in C$ and $m(C)=0$.  An enumeration that omits
zero would instead return the least non-zero norm in $\beta\OO$, namely
$|\Norm\beta|$, and thereby produce a spurious ratio of $1$.  We therefore
test this condition explicitly; see Section~\ref{sec:verify}.  In particular,
if $\beta$ is a unit then $\beta\OO=\OO$, so every $\alpha\in\OO$ is excluded
by this condition.
\end{rem}

We present two methods for finding witness pairs.  Method~I, developed in
Section~\ref{sec:MethodI}, enumerates the residue classes represented by
elements of small norm and searches for an unrepresented class.  Method~II,
developed in Section~\ref{sec:MethodII}, samples residue classes modulo
$\eps^n-1$.  The candidates produced by both methods are passed to the common
certifier of Section~\ref{sec:verify}, which uses the unit action to restrict
an infinite lattice coset to a compact region and then enumerates its lattice
points by slicing and lattice reduction.  In particular, the verification of
a Method~I certificate does not rely on the ideal factorisation and
principality computations used to find it.  Section~\ref{sec:cost} describes
the cost and division of labour between the two methods.  In
Section~\ref{sec:converse} we
independently verify that the $58$ fields of \cite{CL2} are norm-Euclidean,
and Section~\ref{sec:results} combines the two computations to prove
Theorem~\ref{thm:main}.  Section~\ref{sec:files} describes the supplementary
material.

\subsection*{AI usage}
We used Claude Opus~5 to solve this problem through a series of prompts over two days in late July 2026. We subsequently audited the code and wrote the paper with the assistance of GPT~6 Astra. We take full responsibility for verifying all mathematical statements and computations in this paper. 

\section{Method I: residue marking}\label{sec:MethodI}\label{sec:finite}
Method~I rests on an exact enumeration for a fixed modulus.  Fix a non-zero
non-unit $\beta\in\OO$.  We show how to decide whether it admits a witness
class by a finite computation.
Since $K$ has signature $(1,1)$, Dirichlet's unit theorem immediately forces $\rank(\OO^\times)=1$. Moreover, the only roots of unity in $\OO$ are $\pm 1$, so
\begin{equation}\label{eq:units}
   \OO^{\times}=\{\pm\eps^{j}: j\in\Z\}
\end{equation}
for a fundamental unit $\eps$. Fix some $\beta \in \OO$, setting $\fb=(\beta)$ and $R=\OO/\fb$. Denote by $[\gamma]_\fb$ the equivalence class of $\gamma\in \OO$ in $R$. Define
$$
   U=\text{image of }\OO^{\times}\text{ in }R
   =\{[\pm\eps^{j}]_\fb\},
$$
and let $u=|U|$. Clearly $u$ is finite by virtue of $|R|=|\Norm\beta|=: Q$. If
$t=\ord_{\fb}(\eps)$, then $u\in\{t,2t\}$.

\begin{prop}\label{prop:finite}
Assume $h_K=1$, and for each non-zero ideal $\fa\subseteq\OO$ fix a generator
$\gamma_{\fa}$.  Then
\begin{equation} \label{eq: finite classes}
   S:=\{[\gamma]_\fb\;:\;\gamma\in\OO,\ 0<|\Norm\gamma|<Q\}
   =\bigcup_{0<\Norm\fa<Q} U\cdot[\gamma_{\fa}]_\fb,
\end{equation}
which is a union of at most $u\cdot\#\{\fa:\Norm\fa<Q\}$ classes.  Consequently, for
$\alpha\notin\fb$,
\begin{equation} \label{eq: S membership criterion}
   m(\alpha+\beta\OO)\ge Q
   \Longleftrightarrow
   [\alpha]_\fb\notin S .
\end{equation}
\end{prop}

\begin{proof}
If $0<|\Norm\gamma|<Q$ then $\fa=(\gamma)$ is a non-zero ideal of norm
$|\Norm\gamma|<Q$. By \eqref{eq:units}, we have $\gamma = w\gamma_{\fa}$ for some $w\in\OO^{\times}$, so $[\gamma]_\fb\in U\cdot[\gamma_{\fa}]_\fb$.
Conversely, if $\Norm\fa<Q$, every generator $w\gamma_{\fa}$ of $\fa$ satisfies
$|\Norm(w\gamma_{\fa})|=\Norm\fa$. Thus $[w\gamma_{\fa}]_\fb \in S$ and every ideal is
principal, so no ideal is missed, proving~\eqref{eq: finite classes}. If $\alpha\notin\fb$, then $0 \notin \alpha + \beta \OO$. Thus
\begin{align*}
m(\alpha+\beta\OO)<Q \iff & \exists~\gamma \in \OO \text{ st } \vert\Norm(\alpha + \beta \gamma)\vert < Q \\
\iff & \exists~\gamma \in \OO \text{ st } [\alpha + \beta\gamma]_\fb \in S \\
\iff & [\alpha]_\fb \in S,
\end{align*}
with Equivalence~\eqref{eq: S membership criterion} following from the contrapositive statement.
\end{proof}

Everything on the right-hand side of~\eqref{eq: finite classes} is computable exactly:

\begin{itemize}
\item Decompose each rational prime $p<Q$ in $\OO$ and retain the prime
      ideals of norm $<Q$.  Every ideal of norm $<Q$ is a product of these.
      Factoring the defining polynomial modulo $p$ suffices when $p$ does
      not divide the index of the defining order; at index primes one uses
      prime ideal decomposition in the full ring of integers.
\item Generators are multiplicative, $\gamma_{\fa}=\prod_i\gamma_{\fp_i}^{e_i}$
      up to a unit, so only generators of \emph{prime} ideals are needed, and
      only their residues modulo $\fb$.
\item The ring $R$ is represented by the Hermite normal form of $\fb$ in the integral
      basis of $\OO$, so a residue is a triple of integers. Ring
      multiplication uses the (integral) multiplication table of $\OO$.
\end{itemize}
The same enumeration can be continued beyond $Q$.  If all ideals of norm
below $B=\lambda Q$, with $\lambda\ge1$, are enumerated, then each residue
class either acquires a least represented norm below $B$ or is left unmarked.
Any non-zero class whose least represented norm is at least $Q$ is a witness.
Taking $\lambda>1$ costs a factor of roughly $\lambda$, but may sharpen the
resulting lower bound for $M(K)$.

\begin{rem}
We do \emph{not} require $\alpha$ to be invertible in $R$. If $\alpha$ is a
zero divisor, the class $C$ is still a legitimate witness; one only has to enumerate the ideals $\fa$ of norm $<Q$ that are \emph{not} coprime to $\fb$ as well, which our implementation does. Only the class of $0$ has to be excluded,
and no $\gamma$ with $|\Norm\gamma|<Q$ lies in it. Admitting these classes roughly doubles the supply of witnesses for moduli built from small primes and is what makes the small primes above $2$ and $3$ useful.
\end{rem}

\subsection{Counting unmarked classes}
For a fixed modulus, Proposition~\ref{prop:finite} gives a deterministic
algorithm: mark every residue class represented by an element of norm less
than $Q=|\Norm\beta|$, and look for a non-zero unmarked class.  The choice of
$\beta$ is a compromise.  A larger residue ring offers more possible witness
classes, but it also makes the marking step more computationally expensive. By the analytic class number formula for the Dedekind zeta-function of $K$, we have 
\[
   \rho_K=\Res_{s=1}\zeta_K(s)=\frac{2^{r_1}(2\pi)^{r_2}h_KR_K}{w_K\sqrt{|d_K|}}
         =\frac{2\pi R_K}{\sqrt{|d_K|}}\qquad(h_K=1,\ w_K=2),
\]
so that $\#\{\fa:\Norm\fa<X\}\sim\rho_KX$, by the standard ideal-counting theorem.
Proposition~\ref{prop:finite} leads to a deterministic sufficient condition.

\begin{cor}\label{cor:count}
If $1+u\cdot\#\{\fa: 0<\Norm\fa<Q\}<Q$ then a witness class exists for
$\beta$, and $K$ is not norm-Euclidean.
\end{cor}
This follows by counting the classes in $S$ and excluding $[0]_\fb$. Since $\#\{\fa:\Norm\fa<Q\}\approx\rho_KQ$, Corollary~\ref{cor:count} only helps
whenever $u\rho_K<1$, which is rare. The search for a witness class is therefore driven by
the following heuristic.  Approximate the orbits of $U$ acting on $R$ by
$Q/u$ orbits of size $u$, and suppose the $\approx\rho_KQ$ ideals of norm
$<Q$ are distributed uniformly among them.  This ignores smaller orbits
arising from stabilisers, notably on non-invertible classes.  The predicted
number of unmarked orbits is then $Qu^{-1}e^{-\rho_K u}.$
A natural scale for the modulus, with a safety factor $s$, is therefore $|\Norm\beta| = Q \gtrsim sue^{\rho_K u}$.
Two consequences guide the search.  First, the required modulus grows exponentially with $u$: the
modulus must be chosen so that $\eps$ has \emph{small order} modulo $\fb$.
Second, $\rho_K$ is a useful predictor of the difficulty of a
field; see Section~\ref{sec:cost}.

\subsection{Choosing the modulus}
Put $\beta_n=\eps^n-1$ and $\fb_n=(\beta_n)$.  By construction,
$\ord_{\fb_n}(\eps)\mid n$, and $\fp^k\mid\fb_n$ if and only if
$\ord_{\fp^k}(\eps)\mid n$.  We therefore look for products of small prime
power factors of the ideals $\fb_n$.

The points $\xi=\alpha/(\eps-1)$ are fixed by multiplication by $\eps$ on
$K/\OO$, and are natural candidates for large values of $M(K,\xi)$;
see \cite[\S{}Heuristics]{CL1}.  More generally,
$\xi=\alpha/(\eps^n-1)$ is fixed by $\eps^n$.  Allowing $n>1$ is important:
the examples in Section~\ref{sec:results} include maximisers of period two
and three.

For a field with large regulator, the integer
$|\Norm(\eps^n-1)|\asymp e^{nR_K}$ is far too large to factor completely.
However, complete factorisation is unnecessary: only its small rational prime divisors
are needed. Trial division up to a bound $P$, for
$n\in\{1,2,3,4,6,8,12\}$, finds these in a single pass.  We then factor the
corresponding rational primes in $K$ and determine
$\ord_{\fp^k}(\eps)$ directly in $\OO/\fp^k$.

\subsection{Assembling the factors}
The preceding step produces triples
\[
   (\fp^k,\Norm\fp^k,\ord_{\fp^k}(\eps)),
   \qquad \fp^k\mid\fb_n\ \text{for some }n.
\]
We form $\fb=\prod_i\fp_i^{k_i}$ and choose a generator $\beta$ of $\fb$.
Then $\ord_{\fb}(\eps)=\mathrm{lcm}_i\ord_{\fp_i^{k_i}}(\eps)$ and
$|\Norm\beta|=\prod_i\Norm\fp_i^{k_i}$.  For each target $T$ for the lcm, each of the two possible values $u\in\{T,2T\}$, and each of a geometric ladder of safety factors $s$, we solve by branch and bound the small knapsack
\[
   \text{minimise }\textstyle\prod_i\Norm\fp_i^{k_i}
   \quad\text{subject to}\quad
   \textstyle\prod_i\Norm\fp_i^{k_i}\ge sue^{\rho_Ku},
   \qquad \ord_{\fp_i^{k_i}}(\eps)\mid T .
\]
The resulting moduli are deduplicated, ordered by norm, and tried from
smallest to largest.  When several prime ideals of the same norm are
available, we vary the choice: equal norms need not give the same residue-ring
structure.  We also record $n=\ord_{\fb}(\eps)$, so that
$\beta\mid\eps^n-1$; this is the datum needed by the common verification
procedure of Section~\ref{sec:verify}.

\section{Method II: sampling modulo
\texorpdfstring{$\eps^{n}-1$}{epsilon to the n minus 1}}\label{sec:MethodII}

Method~I relies on finding a moderate-size divisor of $\eps^n-1$.  When the
regulator is large, the small prime factors of $|\Norm(\eps^n-1)|$ may simply
be too sparse to provide one.  Method~II removes this obstruction by taking
the whole element  $\beta=\eps^n-1$.
Then $\eps^n\equiv1\pmod\beta$, and the image of the unit group modulo
$\beta$ has size at most $2n$.  This gives a small unit image without any
factorisation; the order of $\eps$ may be a proper divisor of $n$, and
$-1$ may already belong to the subgroup it generates.
The residue ring, however, has size about $e^{nR_K}$ and cannot be
enumerated.  We therefore sample classes $\alpha\pmod\beta$ and apply the
verification of Section~\ref{sec:verify} only to the candidates that survive.

\subsection{The geometry of the unit action}

Write $\xi=(x+iy,z)\in\C\times\R$.
By \eqref{eq:norm-embeddings}, the level set $|\Norm\xi|=1$ is $(x^2+y^2)|z|=1.$
Its two components are surfaces of revolution; the positive component is the
funnel shown in Figure~\ref{fig:norm-one}.  Multiplication by $\eps$ expands
the vertical coordinate by $\eps_{\R}$, contracts horizontal distances by
$\eps_{\R}^{-1/2}$, and rotates through $\arg\eps_{\C}$.  It therefore moves
points along a fixed norm surface.  A fundamental band for multiplication by
$\eps^n$ is
\[
   1\le |z|<\eps_{\R}^n.
\]
This simple picture is the reason that an apparently unbounded minimisation
can be made finite: after restricting to one band, imposing
$|\Norm\xi|\le B$ also bounds the horizontal radius $\sqrt{x^2+y^2}$.

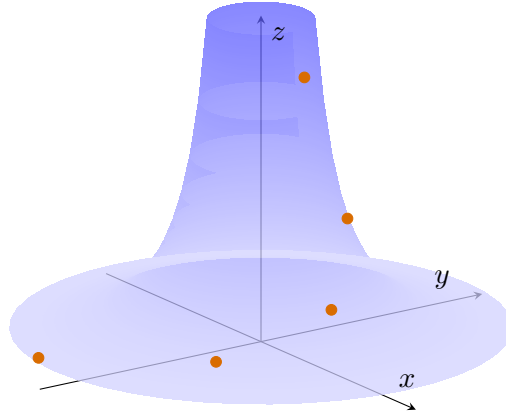
\begin{figure}[t]
\centering
\begin{tikzpicture}
\begin{axis}[
   width=.72\textwidth,
   height=.56\textwidth,
   view={55}{27},
   axis lines=center,
   xlabel={$x$},
   ylabel={$y$},
   zlabel={$z$},
   xmin=-2.2,xmax=2.2,
   ymin=-2.2,ymax=2.2,
   zmin=0,zmax=5.2,
   ticks=none,
   clip=false,
   colormap={normsurface}{color(0cm)=(blue!18); color(1cm)=(blue!65)}]
\addplot3[
   surf,
   shader=interp,
   opacity=.72,
   draw=blue!40,
   domain=0:360,
   y domain=.44:2.05,
   samples=45,
   samples y=24]
   ({y*cos(x)},{y*sin(x)},{1/(y*y)});
\addplot3[
   only marks,
   orange!85!black,
   mark=*,
   mark options={fill=orange!85!black}]
   coordinates {
      (-.3473,-1.9696,.25)
      (.9090,-1.0834,.5)
      (1,0,1)
      (.4545,.5417,2)
      (-.0868,.4924,4)
   };
\end{axis}
\end{tikzpicture}
\caption{The positive component $(x^2+y^2)z=1$ of the absolute norm-one surface,
with the real embedding on the vertical axis.  Multiplication by the
fundamental unit moves a point upward, inward, and around the axis.  The five
marked points schematically represent the consecutive powers
$\eps^k$, $-2\le k\le2$, of a model unit.  The negative component is the
reflection of this surface in the horizontal plane.}
\label{fig:norm-one}
\end{figure}

\subsection{The expected sampling rate}

Sampling $\alpha$ modulo $\beta=\eps^n-1$ is equivalent to sampling the
finite set
\[
   \beta^{-1}\OO/\OO
   =\{\xi\in K/\OO:\eps^n\xi=\xi\}.
\]
These are periodic points of the unit action, not uniformly distributed
points of the ambient torus $(K\otimes_{\Q}\R)/\OO$.  The distinction matters:
the exceptional set in that torus has Haar measure zero.  Indeed, the unit
action is a hyperbolic toral automorphism and hence ergodic.  Extending
$M(K,\xi)$ to $K\otimes_{\Q}\R$ by the same formula, for each $c>0$ the set
$\{\xi:M(K,\xi)<c\}$ is invariant and open and contains a neighbourhood of
zero.  It therefore has full measure, so $M(K,\xi)=0$ almost everywhere.
The proportion of exceptional points in the finite set above can nevertheless
be positive.  We estimate that proportion heuristically as follows.

For $\xi=(x+iy,z)$, let $s=\log|\Norm\xi|$, $t=\log|z|$, and
$\theta=\arg(x+iy)$.  In these coordinates the Euclidean volume element on
$\C\times\R$ is $\tfrac12e^s\,ds\,dt\,d\theta$ on each choice of sign of
$z$.  Hence the part of $\{|\Norm\xi|<1\}$ in one fundamental band for
the $\eps$-action has volume
\[
   2\cdot\frac12
   \int_{-\infty}^0e^s\,ds
   \int_0^{R_K}dt
   \int_0^{2\pi}d\theta
   =2\pi R_K.
\]
For a uniformly translated lattice $\xi-\OO$, the mean number of points
in this bounded band is its volume divided by the covolume
$\sqrt{|d_K|}/2$.  Thus, when $h_K=1$, this continuous average is
\begin{equation}\label{eq:lambda}
   \lambda=\frac{2\pi R_K}{\sqrt{|d_K|}/2}
          =\frac{4\pi R_K}{\sqrt{|d_K|}}
          =2\rho_K.
\end{equation}
For a band of length $nR_K$ the continuous average is $2n\rho_K$.
Using this as an approximation for the finite sampling set, and modelling
the number of encounters in the band by a Poisson random variable, suggests
\begin{equation*}
   \Prob\bigl(m(\alpha+\beta\OO)\ge|\Norm\beta|\bigr)
   \approx e^{-2n\rho_K}.
\end{equation*}
Neither the passage from continuous to discrete sampling nor the Poisson
assumption is exact.  The estimate predicts a preference for $n=1$, with
each increment of $n$ multiplying the work by about $e^{2\rho_K}$, but a
useful witness may first occur at a larger period.  Trials are allocated without relying on a
predicted success probability, in round-robin batches among the unsettled
fields.

\subsection{Cost of a trial}

A failed trial is stopped as soon as one element of
$\alpha+\beta\OO$ with norm below $|\Norm\beta|$ is found.  The volume
calculation above is uniform in $t=\log|z|$.  In the same Poisson model,
when $2n\rho_K$ is large, a rejection examines approximately
\[
   \frac{R_K/\log2}{2\rho_K}
   =\frac{\sqrt{|d_K|}}{4\pi\log2}
\]
of the binary slices introduced in Section~\ref{sec:verify}; this is about
$45$ at the upper end of our discriminant range and is independent of the
regulator.

Three implementation choices keep the work per slice small.  The precision is
proportional to $2nR_K/\log2$ bits.  The large cancellation needed to locate
the coset is performed once per slice; nearby lattice points are then tested
in ordinary floating-point arithmetic.  Finally, trials are grouped into
progressions $\alpha_t=\alpha_0+t\delta$.  The coordinate centre of the coset
in each reduced lattice basis is affine in $t$ and is needed only modulo
$\Z^3$, so the multiprecision calculation is shared by the whole block.  The
step $\delta$ is chosen on the scale of $\beta$: taking $\delta=1$ would move
$\alpha/\beta$ only by $O(e^{-nR_K})$ in the real embedding and would produce
strongly correlated trials.

\section{A common verification procedure}\label{sec:verify}

A certificate from either method records exact defining data for $K$, an
integral basis of $\OO$, a unit $\eps$ with $|\eps_{\R}|>1$, an integer $n\ge1$, and elements
$\alpha,\beta\in\OO$ such that $\beta\mid\eps^n-1$.  Its assertion is
\begin{equation}\label{eq:claim}
   m(\alpha+\beta\OO)\ge |\Norm\beta|.
\end{equation}
For Method~II we have $\beta=\eps^n-1$; for Method~I, $\beta$ is a divisor of
such an element.  One could check a Method~I certificate by running
Proposition~\ref{prop:finite} in reverse: enumerate the ideals of norm below
$|\Norm\beta|$, take their generators, and confirm that none is congruent to
$\alpha$.  We instead use the same verification for both methods.  This
keeps the ideal factorisation, principality and generator computations of
Method~I outside the proof.  It is also essential for Method~II, where
$|\Norm\beta|\asymp e^{nR_K}$ and has $221$ decimal digits in the largest
case, so enumerating all smaller ideals is impossible.  The common verifier
must therefore handle the infinite lattice coset $\alpha+\beta\OO$, of
covolume $|\Norm\beta|\sqrt{|d_K|}/2$, directly.

The following Lemma shows the importance of the unit action to our method.

\begin{lem}\label{lem:inv}
Let $\beta\mid\eps^{n}-1$ and $C=\alpha+\beta\OO$.  Then $\eps^{n}C=C$, and
$|\Norm|$ is constant on the orbits of $\gamma\mapsto\eps^{n}\gamma$ in $C$.
\end{lem}

\begin{proof}
We have
$\eps^{n}\alpha-\alpha=\alpha(\eps^{n}-1)\in\beta\OO$, and
$\eps^{n}\beta\OO=\beta\OO$ because $\eps^{n}$ is a unit; so
$\eps^{n}C=C$.  For the second statement,
$|\Norm(\eps^{n}\gamma)|=|\Norm\eps|^{n}|\Norm\gamma|=|\Norm\gamma|$.
\end{proof}

Neither fundamentality nor the sign convention for $\eps$ is needed here.  Put $\ell=\log|\eps_{\R}|>0$ and
$E=|\eps_{\R}|=e^{\ell}$; thus $\ell=R_K$
when the supplied unit is fundamental.  For $B>0$ define
\begin{equation}\label{eq:compact-band}
   \mathcal F_n(B)
   =\bigl\{(x+iy,z)\in\C\times\R:
       1\le |z|\le E^n,\ (x^2+y^2)|z|\le B\bigr\}.
\end{equation}
This set is compact.  Indeed, the real coordinate is bounded above and away
from zero, while \eqref{eq:norm-embeddings} gives
$x^2+y^2\le B/|z|\le B$.  Every $\xi\in K^\times$ has a translate
$\eps^{kn}\xi$ with $1\le|(\eps^{kn}\xi)_{\R}|<E^n$.  Thus, if
$\alpha\notin\beta\OO$, Lemma~\ref{lem:inv} gives
\begin{equation*}\label{eq:finite}
   m(\alpha+\beta\OO)
   =\min_{\gamma\in(\alpha+\beta\OO)\cap
                   \mathcal F_n(|\Norm\alpha|)}|\Norm\gamma|.
\end{equation*}
The intersection on the right is finite, because it is the intersection of a
compact set with a lattice coset.  This is the finiteness statement behind
the verification.

The volume explains why this reduction is useful.  Repeating the calculation of
\eqref{eq:lambda}, now with $s\le\log B$ and a $t$-interval of length
$n\ell$, gives
$\operatorname{vol}\mathcal F_n(B)=2\pi n\ell B$.  The corresponding
volume-to-covolume estimate is
\[
   \#\{\gamma\in C\cap\mathcal F_n(B)\}\approx
   \frac{2\pi n\ell B}{|\Norm\beta|\sqrt{|d_K|}/2}.
\]
For a fundamental unit, $h_K=1$, and $B=|\Norm\beta|$, this is
$2n\rho_K$, independently of the size of $|\Norm\beta|$.  This is a
heuristic count, not an upper bound for an individual coset; the exhaustive
enumeration below does not rely on it.

\subsection{Slicing}

Although \eqref{eq:compact-band} is compact, its real coordinate ranges over
a factor $E^n=e^{n\ell}$, as large as $e^{509}$ in our computation.  We divide
it into $\lceil n\ell/\log2\rceil$ binary slices
$|z|\in[2^i,2^{i+1})$.  Write
\[
   \iota(\xi)=\begin{pmatrix}\Re\xi_{\C}&\Im\xi_{\C}&\xi_{\R}\end{pmatrix}^{\top}
                =\begin{pmatrix}x&y&z\end{pmatrix}^{\top},\qquad
   a_i=2^{i+1},\qquad b_i=\sqrt{B/2^i},
\]
and set $D_i=\operatorname{diag}(b_i^{-1},b_i^{-1},a_i^{-1})$ and
$\mathbf v=\begin{pmatrix}v_1&v_2&v_3\end{pmatrix}^{\top}=D_i\iota(\xi)$.  Then
\begin{equation}\label{eq:scaled}
   |\Norm\xi| = 2B(v_1^{2}+v_2^{2})|v_3|,
\end{equation}
and a point of slice $i$ has $|v_3|\in[\tfrac12,1)$.  The lattice $\beta\OO$
becomes $D_i\iota(\beta\OO)$, of covolume
$|\Norm\beta|\sqrt{|d_K|}/(4B)$ independently of $i$.  In particular, for
$B=|\Norm\beta|$ the covolume is $\sqrt{|d_K|}/4$.
Let $A_i$ be a basis matrix for this scaled lattice obtained by applying
LLL to an approximation of the initial basis, as described below, and
put $ c_i=A_i^{-1}D_i\iota(\alpha).$
The points of the scaled coset are $A_i(c_i-u)$ with $u\in\Z^3$.

The following Lemma makes the enumeration exhaustive.

\begin{lem}\label{lem:box}
Let $\gamma$ be a point of the coset $\alpha+\beta\OO$ in slice $i$ with
$|\Norm\gamma|<B$, let $\mathbf v=D_i\iota(\gamma)$, and let
$u\in\Z^{3}$ satisfy $\mathbf v=A_i(c_i-u)$.  Then
$\|\mathbf v\|<\sqrt2$ and
\[
   \bigl|u_j-\lfloor (c_i)_j+\tfrac12\rfloor\bigr|
   \le \sqrt{2}\bigl\|(A_i^{-1})_j\bigr\| + \tfrac12
   \qquad (j=1,2,3),
\]
where $(A_i^{-1})_j$ is the $j$-th row of $A_i^{-1}$.
\end{lem}

\begin{proof}
Being in slice $i$ means $|v_3|\in[\tfrac12,1]$, so \eqref{eq:scaled} and
$|\Norm\gamma|<B$ give $v_1^{2}+v_2^{2}<1/(2|v_3|)\le1$, while $v_3^{2}\le1$;
hence $\|\mathbf v\|^{2}<2$.  From $u=c_i-A_i^{-1}\mathbf v$ we get
$|u_j-(c_i)_j|\le\|(A_i^{-1})_j\|\|\mathbf v\|$, and rounding $(c_i)_j$ to the
nearest integer costs a further $\tfrac12$.
\end{proof}

It is therefore enough to enumerate
\[
   |u_j-\lfloor(c_i)_j+\tfrac12\rfloor|\le h_{ij},\qquad
   h_{ij}=\left\lfloor
       \sqrt2\,\|(A_i^{-1})_j\|+\tfrac12
   \right\rfloor,
\]
using a certified upper bound for each row norm.  Taking
$B=|\Norm\beta|$ proves \eqref{eq:claim} if no point in these boxes has
absolute norm below $B$.  The size of the box is part of the proof: omitting a lattice
point can only increase the apparent minimum and could therefore create a
false certificate.

To determine $m(C)$ exactly, rather than merely prove $m(C)\ge|\Norm\beta|$,
the verifier also needs a point attaining the claimed minimum.  If the first
sweep finds a least norm $v\le B$, exhaustiveness below $B$ proves
$m(C)=v$.  Otherwise it repeats the enumeration with
\[
   B'=\min\{v+1,|\Norm\alpha|+1\},
\]
or with $B'=|\Norm\alpha|+1$ if the first sweep found no point.  A known
coset point then has norm strictly below $B'$, so its unit translate lies in
the region covered by the second sweep.  The least norm found in that sweep
is therefore the exact minimum.  We use this procedure for the quantitative
results in Section~\ref{sec:results}.

\subsection{Search and certification}
The search and the proof have independent implementations.  Method~I uses
ideal factorisation and generators to find candidates, whereas Method~II
screens candidates using approximate arithmetic.  Neither computation is
reused in certification.

The common verifier is implemented in C++, using FLINT~\cite{FLINT} for exact
integer and rational arithmetic, Arb~\cite{Arb} for certified root isolation
and real and complex interval arithmetic, and fplll~\cite{fplll} for lattice
reduction.  The underlying multiprecision arithmetic uses GMP~\cite{GMP}
and MPFR~\cite{MPFR}.  Its geometric computations are enclosed as follows:
\begin{enumerate}
\item certified root enclosures determine the embeddings;
\item an enclosure of $A_i^{-1}$ gives upper bounds for its row norms and
      hence for $h_{ij}$ in Lemma~\ref{lem:box};
\item if an enclosure of $(c_i)_j$ allows rounded centres from $a_j$ through
      $b_j$, the enumeration covers the whole interval
      $a_j-h_{ij}\le u_j\le b_j+h_{ij}$;
\item every point in the enumeration box is tested by an exact integer norm.
\end{enumerate}
If the inverse cannot be enclosed or the resulting box is too large,
verification is refused.  Insufficient precision can therefore prevent
certification, but cannot produce a false certificate.  No slice was refused
in the computation of Section~\ref{sec:results}.

LLL is used only heuristically in certification: no guarantee of reducedness
is required from its implementation.  If $V_i$ is the initial basis matrix
of the scaled lattice, we apply LLL to a scaled and rounded approximation
of $V_i$ to obtain an integer matrix $T$.  We check $|\det T|=1$ exactly,
then form $A_i=V_iT$ and enclose its inverse in interval arithmetic.
The enumeration bounds are derived directly from this basis by
Lemma~\ref{lem:box}, independently of how $T$ was obtained.  A poor choice
of $T$ can make enumeration expensive or cause a refusal, but cannot
invalidate a certificate.

Lattice reduction is nevertheless crucial to efficiency.  Even after the
slice scaling, the basis inherited from $\beta\OO$ can be highly skewed,
giving large row norms of its inverse and hence a prohibitively large
enumeration box.  A good change of basis makes these bounds small enough
for exhaustive enumeration to be practical; it is this actual quality of
the basis, rather than any certified LLL property, that matters.

As noted in Remark~\ref{rem:nonzero}, the verifier must also check
$\alpha\notin\beta\OO$.  It does so exactly, by dividing in $K$ and testing
whether the quotient belongs to $\OO$.

\subsection{Inputs and exact checks}

The verifier takes the defining polynomial, a basis, a unit, and the witness
as exact algebraic data.  It verifies that the basis contains $1$ and is closed
under multiplication, but it cannot verify that the resulting order is the
full ring of integers.  A correct basis of $\OO_K$ is therefore the only
trusted external input.  The supplied unit need not be fundamental.

We compute the basis independently in PARI/GP~\cite{PARI} and
Magma~\cite{Magma}.  A basis of $\OO_K$
is not unique, so the two need not be identical.  We check that they span
the same lattice, i.e.\ that the change of basis $T$ with $B=TA$ is integral
with $|\det T|=1$.  Over the $17283$ fields the two spans agree in every case.

The certifier makes the following exact checks.

\begin{enumerate}
\item \textbf{The basis spans an order.}  Its multiplication table is
      constructed from the basis by exact rational arithmetic and verified to be
      integral, and $\omega_1=1$.
\item  \textbf{That $\eps$ is a unit}: $|\Norm\eps|=1$, an exact integer determinant.
\item \textbf{The coset is $\eps^{n}$-stable}: $\beta\mid\eps^{n}-1$, an
      exact rational solve.  This is the hypothesis of Lemma~\ref{lem:inv}, on
      which the reduction to a bounded region rests.
\item \textbf{That $\alpha\notin\beta\OO$}, the hypothesis of
      Proposition~\ref{prop:crit} (Remark~\ref{rem:nonzero}).
\item \textbf{The reduction is a change of basis}: $|\det T|=1$, exactly.
\item \textbf{Every norm entering a decision} is an exact integer.
\end{enumerate}

It also verifies $|\eps_{\R}|>1$ by a certified enclosure, replacing the
supplied unit by its inverse if necessary before forming the modulus or
checking the divisibility condition.  A torsion unit, or an enclosure that
does not establish expansion, is refused.  Neither GRH nor
\texttt{bnfcertify} enters this verification; neither the class number nor
any ideal factorisation enters the certificate.

\begin{rem}\label{rem:hdir}
The classification uses the complete enumeration of complex cubic fields in
the discriminant range and the unconditional class-number computations
described in the introduction.  These establish the completeness of the
list of fields with $h_K=1$ that we examine.

Verification of an individual certificate requires less: given a correct
basis of $\OO_K$, it uses neither the class number nor the regulator, and
checks directly that the supplied $\eps$ is a unit without requiring it to
be fundamental.  This separation keeps the verifier's external inputs to a
minimum; it does not reflect any uncertainty about the class numbers or
regulators, which can be computed rigorously and readily at these small
discriminants.
\end{rem}

\section{Cost and shape of the computation}\label{sec:cost}

For Method~I the heuristic modulus scale grows as $ue^{\rho_Ku}$, where
$u$ is the size of the unit image.  A small image, often of size two, is
therefore desirable.  For Method~II, the model with $n=1$ predicts
$e^{2\rho_K}$ trials, with a typical rejection examining
$O(\sqrt{|d_K|})$ slices.  Both estimates depend exponentially on
\[
   \rho_K=\Res_{s=1}\zeta_K(s)=\frac{2\pi R_K}{\sqrt{|d_K|}}\qquad(h_K=1),
\]
which measures the density of ideals in the ideal-counting asymptotic.
It is a useful predictor of search cost, but does not determine the
Euclidean minimum or the success probability for an individual field.
Its distribution over the $17283$ fields is given below.

\begin{table}[!h]
\caption{Distribution of $\rho_{K}$}
\begin{center}
\begin{tabular}{lrrrrrrrrr}
\toprule
$\rho_K$ & $<1$ & $1$--$2$ & $2$--$3$ & $3$--$4$ & $4$--$5$ & $5$--$6$ & $6$--$7$ & $7$--$8$ & $>8$\\
\midrule
fields & 4153 & 6618 & 3971 & 1594 & 629 & 217 & 76 & 24 & 1\\
$\%$ above & 76.0 & 37.7 & 14.7 & 5.5 & 1.8 & 0.58 & 0.14 & 0.01 & ---\\
\bottomrule
\end{tabular}
\end{center}
\end{table}

The methods are complementary.  Method~I is effective when a small modulus
with a small unit image is available.  Method~II avoids finding such a
modulus, but its success depends on the exceptional points among those fixed
by $\eps^n$; the smallest period need not contain a witness.  Our search
therefore applies Method~I with increasing bounds and work budgets, followed
by sampling for the unresolved fields.  The final certificate set described
in Section~\ref{sec:results} contains $16825$ certificates from Method~I and
$400$ from Method~II.  These are the certificates retained for the proof,
not counts of all successful searches.

\subsection{The extreme case:
\texorpdfstring{$d_K=-152135$}{discriminant -152135}}\label{sec:extreme}

The field of discriminant $-152135$ provides the sharpest test of the cost
model.

When $d_K=-152135$ we have $R_K\approx508.8$ and $\rho_K\approx8.20$, the largest value of $\rho_K$
among all $17283$ fields, the next being $\approx7.82$.
The median over all $17283$ fields is $\approx1.64$.
Section~\ref{sec:MethodII} predicts a success rate $e^{-2\rho_K}\approx e^{-16.4}$, so
about $1.3\times10^{7}$ trials, and Section~\ref{sec:verify} predicts
$\lceil R_K/\log2\rceil=735$ slices of which a rejection
inspects about $\sqrt{|d_K|}/(4\pi\log2)\approx45$ under the heuristic
model.  In contrast, the Method~I modulus scale
$2e^{2\rho_K}\approx2.6\times10^7$ is much larger than the moduli retained
in our Method~I certificates.

The final certificate uses $n=1$, with $|\Norm\beta|$ a $221$-digit integer,
and proves
\[
   M(-152135)\ge1.038709.
\]
The displayed bound is rounded down from the exact rational value.  The
stored coordinates of $\alpha$ have at most $74$ digits: choosing a small
representative of the coset reduces the certificate's size without changing
the bound.  This remains only a lower bound for $M(K)$, not a determination
of the field's Euclidean minimum.

\section{Independent verification for the norm-Euclidean fields}\label{sec:converse}

Methods~I and~II cannot prove norm-Euclideanity.  We verify this direction
by an independent implementation of the covering method of 
\cite{CL1,CL2}, using rigorous interval bounds throughout.

The object is a fundamental domain for $K\otimes_{\Q}\R$ modulo $\OO_K$, taken as
the unit cube in the coordinates of an integral basis.  A cube $T$ is
\emph{covered} when a single $\gamma\in\OO_K$ satisfies
$\sup_{\xi\in T}|\Norm(\xi-\gamma)|<1$; uncovered cubes are subdivided into
eight and retried.  If every point of the fundamental domain has a strict
norm $<1$ witness, the resulting open cover has a finite subcover, so
sufficiently fine subdivision can in principle cover the whole domain.
To reduce the subdivision required, we interleave covering with elimination
under the unit action, as in \cite{CL2}.  Since $M(K,\xi)$ is invariant
under multiplication by $\eps$, a cube whose image $\eps T$, reduced modulo
$\OO$, meets no remaining  and cannot contain an exceptional point.
Covering and elimination are repeated to a fixed point at each depth.

Everything in the proof is exact or enclosed.  The defining polynomial,
integral basis, and unit are given by exact integer and rational data.
Certified root enclosures determine the embeddings, the cube coordinates
are exact dyadics, and interval arithmetic bounds the supremum on each cube.
The elimination step uses integer arithmetic.

The covering prover uses FLINT~\cite{FLINT} for exact arithmetic and
Arb~\cite{Arb} for the root enclosures.  We carry out the covering computation
with three implementations of the interval arithmetic: Arb balls, MPFI
endpoint intervals~\cite{MPFI}, and Platt's \texttt{int\_double}
library~\cite{Platt} using double-precision endpoints with directed rounding.
All three prove norm-Euclideanity for all $58$ fields.  They share
the exact algebra and unit-elimination code, so their agreement provides an
independent check of the interval arithmetic rather than of the entire
implementation.

\subsection{A mean-value bound}

Write $\xi-\gamma=(x+iy,z)$, so that
$f(\xi)=|\Norm(\xi-\gamma)|=(x^2+y^2)|z|$.  Here $x,y,z$ are affine
functions of the integral-basis coordinates $t_1,t_2,t_3$ of $\xi$.
Interval evaluation encloses their ranges on a cube, including
the uncertainty in the embedding coefficients.  Forming the product
$(x^2+y^2)|z|$ from these separate enclosures loses the dependence between
the three functions and can overestimate the supremum by a term of first
order in the cube width.

When $z$ has constant nonzero sign on the cube $T$, we also use the mean-value
bound
\[
  \sup_{T}f \le f(m)+r\sum_{j}\sup_{T}\Bigl|\frac{\partial f}{\partial t_j}\Bigr|,
\]
with $m$ the (exactly dyadic) centre, $r$ the half-width and every supremum
again taken as an enclosure.  The derivative is evaluated as a signed sum
before taking its absolute value, so cancellation between the variation of
$x^2+y^2$ and that of $|z|$ is retained.  For a fixed smooth branch, this
bound differs from the true supremum by $O(r^2)$ in exact arithmetic; the
variation of $f$ itself is generally still first order.  A maximiser of the
lower envelope $M(K,\cdot)$ need not be a stationary point of any individual
translated norm function.  If the enclosure of $z$ contains zero, we use the
product bound instead; otherwise we take the smaller of the two bounds.

The sharper enclosure also makes it practical to test witnesses $\gamma$
far from the cube in integral-basis coordinates.  Such elements can still
have small norm difference because one embedding is small, despite the
others being large.

\subsection{Enumerating the unit image of a cube}

The elimination step asks which cubes the image $\eps T$ meets.  Taking the
axis-aligned bounding box of that image can be expensive: $|\Norm\eps|=1$, so
the image parallelepiped has the volume of a single cell, but the expanding
embedding shears it into a long, thin sliver.  It meets a number of cells of the
order of the sum of its extents while its bounding box holds their
product: for $d_K=-744$, about $10^{4}$ against $5\times10^{5}$.

In grid coordinates, the image of the cell $c+[0,1]^{3}$ is
$M_{\eps}c+M_{\eps}[0,1]^{3}$ with $M_{\eps}c$ an integer vector, so every
cube's image is a translate, by a lattice vector, of one fixed parallelepiped:
the set of cells met is the same for every cube.  It is computed once, by the
separating-axis theorem.  For two parallelepipeds, it is enough to test the
face normals of each, together with the cross products of one edge direction
from each.  Here the cell has normals and edges $e_j$, and the parallelepiped
has edges the columns of $M_{\eps}$ and, since $|\det M_{\eps}|=1$ makes the
inverse integral, normals the rows of $M_{\eps}^{-1}$; every axis and every
projection is therefore an integer, and the test is exact.  Touching counts as
meeting, so the set is a superset of the cells really met: an extra cell can
only leave a cube in play, which is the safe direction, whereas a missing cell
would permit a wrong elimination.

We find this set without enumerating the bounding box, by flood filling from the
cell holding the centre of the parallelepiped.  We regard two cells as
neighbours when they share a face, edge, or vertex (the $26$ neighbours of a
cell in a three-dimensional grid).  The cells meeting a convex body are
connected in this graph: a line segment between points in two such cells stays
inside the body and passes only through cells that meet it.  The enumeration is
therefore proportional to the answer
rather than to the box; for $d_K=-744$ the set has $10616$ cells.

All $58$ fields are proved norm-Euclidean, with maximum subdivision depth
$12$.

\section{Computational results}\label{sec:results}

We use one final set of $17225$ certificates, supplied with the supplementary
material~\cite{Code}.  Every certificate has been checked by the common verifier of
Section~\ref{sec:verify}, including an exact determination of its coset minimum.  All certificate counts,
bounds, and sizes reported here refer to this set.  The percentages below
are relative to the $17283$ input fields.

On a laptop with an Intel Core i7-1265U processor, using twelve parallel
workers, verification of all certificates, including their exact coset
minima, took about $4.5$ minutes of wall time, while the Platt implementation
proved norm-Euclideanity for all $58$ fields in about $7$ seconds.

\begin{table}[!h]
\caption{Certificate results}
\begin{center}
\begin{tabular}{lrr}
\toprule
certificates from Method~I (residue marking) & $16825$ & $97.35\%$\\
certificates from Method~II (sampling)       & $400$   & $2.31\%$\\
\midrule
total settled                                       & $17225$ & $99.66\%$\\
remaining                                           & $58$    & $0.34\%$\\
\bottomrule
\end{tabular}
\end{center}
\end{table}

\noindent The $58$ cases that remain are exactly the discriminants conjectured in \cite{CL2} to be norm-Euclidean, the
largest being $|d_K|=876$.  Every complex cubic field with $h_K=1$ and
$|d_K|\le170520$ outside that list is proved not norm-Euclidean, with
an exact certificate for each.

\begin{proof}[Proof of Theorem~\ref{thm:main}]
If $K$ is norm-Euclidean then $\OO_K$ is a principal ideal domain, so $h_K=1$,
and $|d_K|\le170520$ by \eqref{eq:vdl}.  The complex cubic fields with
$h_K=1$ in that range are exactly the $17283$ of the LMFDB list~\cite{LMFDB},
whose completeness there is unconditional.  Each of the $17225$ of them outside
the list of \cite{CL2} carries a certificate $(\alpha,\beta)$ with
$m(\alpha+\beta\OO_K)\ge|\Norm\beta|$, so $M(K)\ge1$ and $K$ is not
norm-Euclidean by Proposition~\ref{prop:crit}.  Conversely each of the $58$ is
norm-Euclidean by the covering computation of Section~\ref{sec:converse},
which reproves \cite{CL1,CL2} independently.
\end{proof}

Among the Method~I certificates, $1499$ have
$m(\alpha+\beta\OO)/|\Norm\beta|=1$.  This is an exact value for the
chosen coset, not a claim that $M(K)=1$. The remaining $15726$ certificates
give bounds strictly greater than one. The $400$ Method~II fields have $\rho_K$ from $0.76$ to $8.20$ (median $4.11$) and $R_K$ from $2.2$ to $508.8$ (median $197.1$).  Of these, $398$ use $n=1$, one uses $n=2$ ($d_K=-327$), and one uses $n=3$ ($d_K=-680$). Their certified ratios range from approximately $1.00000859$ to $2.0845$, with median $1.0921$.

The representatives $\alpha$ are size-reduced within their cosets, rather
than chosen to attain the minimum.  Their coordinates have at most four
decimal digits for Method~I and $74$ for Method~II, excluding signs.
This choice affects the storage size, not the value of the certificate.

\subsection{Checks against known minima}
Several fields provide useful comparisons with earlier computations.

\begin{itemize}
\item For $d_K=-327$, the Method~II certificate has ratio $101/99$, using
      $n=2$, matching the known value of $M(K)$ in \cite{CL2}.
      The maximiser is fixed by $\eps^2$ and is not visible for odd $n$.
\item For $d_K=-680$, the Method~II certificate, with $n=3$, has ratio
      $81956632/81182612$, matching the known minimum
      in \cite{CL1}.
\item For $d_K=-1055$, the Method~I certificate reproduces the bound
      $M(K)\ge1483/1370$ \cite{CL1}.
\end{itemize}
In the first two examples, our certificates supply the lower bounds;
equality with $M(K)$ also uses the upper bounds established in the cited work.

\section{Supplementary material} \label{sec:files}

The  repository~\cite{Code} contains the field list, one certificate for
each of the $17225$ non-norm-Euclidean fields, and the covering results for
the remaining $58$.  Each certificate includes the exact algebraic data
needed for verification by the method of Section~\ref{sec:verify}.
Source code, data formats, and instructions for reproducing the computations
are documented in the repository's README.

\section*{Acknowledgments}
The authors thank Bryce Kerr and Valeriia Starichkova for fruitful discussions on the topic. They would also like to express their gratitude to Kevin McGown for carefully reading an earlier version of the manuscript and for his helpful suggestions.

\end{document}